\documentclass[11pt,oneside]{amsart}
\usepackage[margin=1.33in]{geometry}

\pdfoutput=1

\usepackage{amsmath, amscd}
\usepackage{amsfonts}
\usepackage{amssymb}
\usepackage{amsthm}
\usepackage{upref}
\usepackage{multirow}
\usepackage{graphicx}
\usepackage{subfig}
\usepackage{url}
\usepackage{mathtools}
\usepackage{pinlabel}
\usepackage{color}
\usepackage{hyperref}
\usepackage{changepage}
\usepackage{epstopdf}
\usepackage{tikz}
\usepackage{enumerate}
\usepackage{hyperref}

\newcommand\quotient[2]{\raise1ex\hbox{$#1$}\Big/\lower1ex\hbox{$#2$}}

\newcommand{\Z}{\mathbb{Z}}

\newcommand{\Out}{\operatorname{Out}}
\newcommand{\Aut}{\operatorname{Aut}}

\newcommand{\outo}[1]{\ensuremath{\operatorname{Out}^0\left({#1}\right)}}
\newcommand{\outoW}{\outo{\cox{\Gamma}}}
\newcommand{\cox}[1]{\ensuremath{W_{#1}}}
\newcommand{\st}{\ensuremath{\operatorname{st}}}

\newcommand{\lk}{\ensuremath{\operatorname{lk}}}
    \newcommand{\link}[1]{\ensuremath{\lk\left({#1}\right)}}

\newtheorem{thm}{Theorem}[section]
\newtheorem{thmspecial}{Theorem}
\newtheorem{corspecial}[thmspecial]{Corollary}

\newtheorem{lem}[thm]{Lemma}
\theoremstyle{remark}
\newtheorem{rem}[thm]{Remark}
\theoremstyle{definition}
\newtheorem{defn}[thm]{Definition}
\theoremstyle{plain}
\newtheorem{prop}[thm]{Proposition}
\theoremstyle{remark}

\theoremstyle{definition}

\theoremstyle{definition}

\theoremstyle{plain}

\theoremstyle{definition}

\newcounter{claimcounter}[thm]
\numberwithin{claimcounter}{thm}
\newcounter{casecounter}

\newenvironment{case}{\stepcounter{casecounter}{\noindent\textit{Case \thecasecounter.}}}{\vspace{0.1cm}}
\newcounter{subcasecounter}[casecounter]
\numberwithin{subcasecounter}{casecounter}

\newcommand{\mids}{\, | \, }
\newcommand{\m}{\ensuremath{^{-1}}}

\renewcommand{\phi}{\varphi}

\renewcommand{\case}[1]{\smallskip \noindent {\textbf{Case #1.}}}

\renewcommand{\Im}{\operatorname{Im}}

\title[Out of a RACG is either large or virtually abelian]{Outer automorphism groups of right-angled Coxeter groups are either large or virtually abelian}

\author{Andrew Sale}
\author{Tim Susse}

\begin{document}
    \begin{abstract}
  	We generalise the notion of a separating intersection of links (SIL)
    to give necessary and sufficient criteria on the defining graph $\Gamma$ of a right-angled Coxeter group $W_\Gamma$ so that its outer automorphism group is large:
    that is, it contains a finite index subgroup that admits the free group $F_2$ as a quotient.
    When $\Out(W_\Gamma)$ is not large, we show it is virtually abelian. We also show that the same dichotomy holds for the outer automorphism groups of graph products of finite abelian groups.
    As a consequence, these groups have property (T) if and only if they are finite, or equivalently $\Gamma$ contains no SIL.
    \end{abstract}

\maketitle

Given a simplicial graph $\Gamma$, with vertex set $V(\Gamma)$,
we can define its associated right-angled Coxeter group (RACG) $W_\Gamma$ by considering the group with generating set $V(\Gamma)$, 
where the list of relators is that each vertex is taken to be an involution, and two vertices commute in $W_\Gamma$ if an only if they are adjacent in $\Gamma$.
The geometry and algebra of (right-angled) Coxeter groups has long been studied in many areas of Mathematics (see \cite{Davis:book} for a comprehensive reference). One of the central themes of this work has been using graph theoretic properties of $\Gamma$ to understand the group $W_\Gamma$, see \cite{BehrstockHagenSisto:coxeter, Moussong:thesis} for particular examples.

Here we continue this trend by studying $\Out(W_\Gamma)$, the outer automorphism group of a RACG, and, more generally, outer automorphism groups of graph products of finite abelian groups. 
Outer automorphism groups are of general interest to many geometric group theorists, with a significant amount of attention paid to mapping class groups (outer automorphism groups of closed surface groups), $\Out(F_n)$, and there is growing interest in outer automorphism groups of right-angled Artin groups. 
Previously, some work has been done on $\Out(W_\Gamma)$, focussing mainly on generation and presentations for the group \cite{Tits:Coxter, Muhlherr:presentation, Laurence-thesis}. More recently, Gutierrez, Piggott, and Ruane studied structural properties of outer automorphism groups of graph products of finite abelian groups in \cite{GPR_automorphisms}, and we generalize their work here.

In \cite{GPR_automorphisms}, separating intersections of links (SILs) play a crucial role. They can be defined as triples of vertices, which we denote $(x,y\mids z)$, and they determine when specific automorphisms (partial conjugations) do not commute. Gutierrez, Piggott, and Ruange showed that $\Out(W_\Gamma)$ is infinite if and only if $\Gamma$ contains a SIL.
We take this one step further, defining two variations on a SIL, which we call STILs and FSILs, and show that these characteristics of $\Gamma$ determine a dichotomy between $\Out(W_\Gamma)$ being virtually abelian, and it taking on a somewhat opposite property---that of being large.
We recall that a group $G$ is called \emph{large} if it contains a finite index subgroup $G_0$ with an epimorphism of $G_0$ onto $F_2$, the free group of rank two. 
The definitions of SILs, STILs and FSILs are given in Section \ref{sec:SILs STILS FSILS}.

\newpage

Our main theorem is the following:

\begin{thmspecial}\label{thmspecial:main}
   Let $\Gamma$ be a finite simplicial graph. Then the following are
   equivalent:
      \begin{enumerate}
        \item $\Out(W_\Gamma)$ is large;
        \item $\Out(W_\Gamma)$ is not virtually abelian;
        \item $\Gamma$ contains a STIL or an FSIL.
      \end{enumerate}
In particular, $\Out(W_\Gamma)$ obeys a strict dichotomy: it is
either virtually abelian or large.
\end{thmspecial}

This is a special case of a more general result we prove that applies to the outer automorphism groups of graph products of finite abelian groups.
For each vertex $v$ of a simplicial graph $\Gamma$, let $G_v$ be a group.
We define the \emph{graph product} of the family $(\Gamma, \{G_v : v\in V(\Gamma)\})$ to be the quotient of the free product $\ast_{v\in V(\Gamma)} G_v$ determined by adding the relators that $G_u$ and $G_v$ commute whenever $u$ and $v$ are adjacent in $\Gamma$.

We describe a special family of graph products that includes RACGs and RAAGs.
Let $p$ be a map that assigns to each vertex of $\Gamma$ a prime power or $\infty$,
thus forming a labelled graph $(\Gamma,p)$.
The associated graph product $G(\Gamma,p)$ is defined, in a similar manner to a RACG, with generating set $V(\Gamma)$, and a defining set of relators given by:
\begin{itemize}
	\item  $v^{p(v)}=1$ for each $v\in V(\Gamma)$ such that $p(v) \neq \infty$,
	\item  $[u,v]=1$ whenever $u$ and $v$ are adjacent in $\Gamma$.
\end{itemize}

We suppress the $p$ from the notation, writing $G(\Gamma,p) = G_\Gamma$, when it is clear that $\Gamma$ comes equipped with such a map.
We will write $p<\infty$ when $p(v)\neq \infty$ for each $v\in V(\Gamma)$.

The group $G(\Gamma,p)$ is a graph product where vertex groups are either $\Z$ or $\Z / p(v)\Z$. 
Note that any graph product of (finite) abelian groups is isomorphic to such a group (with $p<\infty$), using the structure theorem for finite abelian groups to directly decompose each vertex group. 
We call this $G(\Gamma,p)$ the \emph{standard graph presentation} of such a group.

When we allow vertices to have order greater than 2, certain SILs will result in a large outer automorphism group. Specifically, if $(x,y\mids z)$ is a SIL and either $p(x)>2$ or $p(y)>2$, then we call the SIL a \emph{non-Coxeter SIL}. The following extends Theorem~\ref{thmspecial:main}.

\begin{thmspecial}\label{thmspecial:main graph prod}
    Let $(\Gamma,p)$ be a finite simplicial labelled graph with $p<\infty$. Then the following are equivalent:
      \begin{enumerate}
        \item $\Out(G_\Gamma)$ is large;
        \item $\Out(G_\Gamma)$ is not virtually abelian;
        \item $\Gamma$ contains a STIL, an FSIL, or a non-Coxeter SIL.
      \end{enumerate}
     In particular, the outer automorphism group of a graph product of finite abelian groups obeys a strict dichotomy: it is
     either virtually abelian or large.
\end{thmspecial}

One way to interpret Theorems \ref{thmspecial:main} and \ref{thmspecial:main graph prod} is that at one extreme $\Out(G_\Gamma)$ will have many quotients 
(being large implies that any finitely generated group will be a quotient of some finite index subgroup),
while at the other it will have a heavily restricted family of quotients.

Comparing this to the situation for a right-angled Artin group (RAAG)
$A_\Gamma$, Guirardel and the first-named author
\cite{GuirardelSale-vastness} showed that the groups $\Out(A_\Gamma)$ observe a similar, but slightly less clear-cut, dichotomy when comparing the supply of quotients.
There are
several ways to describe it, but one way is as follows.
Either $\Out(A_\Gamma)$ has all finite groups involved---meaning for every finite group $H$ there is a finite index subgroup of $\Out(A_\Gamma)$ that admits $H$ as a quotient---or $\Out(A_\Gamma)$ has a finite index subgroup that admits a quotient by a finitely generated nilpotent subgroup that is isomorphic to the direct product of finitely many copies of $\operatorname{SL}(n_i,\Z)$, where each $n_i\geq 3$.

We note that the question of whether $\Out(F_n)$ is large is still open for $n\geq 4$, however they do have all finite groups involved, a consequence of the representations described by Grunewald and Lubotzky \cite{GrunewaldLubotzky:AutFree}.
This is crucial in the RAAG case, but it does not come up in the RACG case, and we are able to show they observe the stronger dichotomy.

The nilpotent normal subgroup and arithmetic quotient described above come about through the properties of transvections on $A_\Gamma$.
When dealing with RACGs, up to passing to a finite index subgroup of $\Out(W_\Gamma)$, we can ignore transvections and just focus on partial conjugations.
This allows us to ``tidy up'' the dichotomy for RACGs.
Indeed, in \cite{GuirardelSale-vastness}, SILs are used to show $\Out(A_\Gamma)$ is large in cases when transvections do not cause (potential) obstructions (see Section~\ref{sec:gens} for definitions).

Since infinite abelian groups and free groups do not have Kazhdan's
property (T), and it is a property preserved under taking finite index subgroups and quotients (see
\cite[Chapter 1]{PropT}), Theorem~\ref{thmspecial:main graph prod} implies that
$\Out(G_\Gamma)$ can have Property (T) only when it is a finite
group. Combining with \cite[Theorem 1.4]{GPR_automorphisms}, this gives the following.

\begin{corspecial}
    Let $G_\Gamma$ be a graph product of finite abelian groups. Then, $\Out(G_\Gamma)$ has Kazhdan's property (T) if and only if it is finite. 
    
    Furthermore, if $G_\Gamma$ has the standard graph presentation, then $\Out(G_\Gamma)$ has property (T) if and only if $\Gamma$ has no SIL.\end{corspecial}

Obtaining such a precise statement is in sharp contrast to other commonly studied outer
automorphism groups. In particular, largeness and property
(T) are unknown for $\Out(F_n)$ when $n\ge 4$, for the outer
automorphism groups of many RAAGs, and for mapping
class groups, though there are some partial results, see for
example
\cite{Aramayona-MartinezPerez,GuirardelSale-vastness,GrunewaldLubotzky:AutFree,
GLLM:MCGinvolve}.

\medskip
The structure of the paper is as follows.
Section \ref{sec:prelims} contains the necessary definitions, as well as some preliminary results concerning the structure of graphs with STILs and FSILs.
Theorem~\ref{thmspecial:main} is proved in Sections \ref{sec:v ab} and \ref{sec:large case}.
In the former we show that in the absence of a STIL or an FSIL the commutator subgroup of $\Out^0(W_\Gamma)$, the group of outer automorphisms generated by partial conjugations---a finite index subgroup of $\Out(W_\Gamma)$---is abelian and give a finite generating set for it. Since partial conjugations are involutions, this implies $\Out(W_\Gamma)$ is virtually abelian.
In Section \ref{sec:large case}, we use factor maps to STILs and FSILs to obtain homomorphisms from $\Out^0(W_\Gamma)$ to virtually free groups, implying largeness.
Finally, Section \ref{sec:graph products} describes how the proof generalizes to graph products of finite abelian groups, obtaining Theorem~\ref{thmspecial:main graph prod}.

\section{Preliminaries}\label{sec:prelims}

\subsection{Conventions}

For commutators, we will write $[a,b]=aba\m b\m$.
We apply maps on the left, so given $f,g : X\to X$, the composition $fg(x)$ is given by $f(g(x))$.

\subsection{Generating the automorphism groups}\label{sec:gens}

In his unpublished thesis \cite{Laurence-thesis}, Laurence
investigated automorphism groups of graph products. He also looked
at two special cases, those of RAAGs (see also \cite{Laur95}), and
those of graph products of finite cyclic groups, where each vertex
group has the same order. Amongst the latter class of groups we have
RACGs. In particular, Laurence's work confirms a conjecture of
Servatius \cite{Servatius} that $\Aut(W_\Gamma)$ is generated by
three different types of elements.

There are \emph{graph symmetries}: these permute the vertices of
$\Gamma$ while preserving the graphical structure.
There are \emph{transvections}: for pairs of adjacent vertices $v,w$ in $\Gamma$
such that $\st(v)\subseteq \st(w)$, we can define an automorphism
that sends $v$ to $vw$ and fixes all other vertices. The third type
are partial conjugations (also known as locally inner
automorphisms).

\begin{defn}
    For each $v\in\Gamma$ and $C$ a connected component of $\Gamma\setminus\st(v)$, we
    define the \emph{partial conjugation} with multiplier $v$ and support $C$ on
    the vertex set of $\Gamma$:
    $$\chi^{v}_{C}(w)=\begin{cases} vwv, & \text{ if } w\in C\\
    w, & \text{ otherwise}\end{cases}.$$
This extends to an automorphism
    of $G_{\Gamma}$. Furthermore, it is not inner exactly when $\st(v)$ is
    separating in $\Gamma$ and $C\neq\Gamma\setminus\st(v)$.
\end{defn}

We also use this notation to describe partial conjugations where the multiplier is not necessarily a vertex of $\Gamma$.
For $g\in G_\Gamma$ and $C\subseteq\Gamma$, write $\chi^g_C$ for the automorphism defined by sending $w\in C$ to $gwg^{-1}$, and fixing all other vertices, if such an automorphism exists.

When dealing with
graph products $G(\Gamma,p)=G_\Gamma$ 
a generating set for $\Aut(G_\Gamma)$ was given by Corredor and Gutierrez \cite{Corredor-Gutierrez}.
Graph symmetries need to preserve the labelling, while we also need to add automorphisms which act only on one vertex group, fixing all others.
For this paper, however, we do not need them. 

It is well-known, and not hard to see, that the subgroup
$\Aut^0(W_\Gamma)$ generated by the partial conjugations has finite
index in $\Aut(W_\Gamma)$. The following says this is also true for the graph products considered in this paper.

\begin{prop}[{\cite{Muhlherr:presentation}, \cite[Theorem 1.2, Lemma 2.8, Remark 2.10]{GPR_automorphisms}}]
	Let $G_\Gamma$ be a graph product of finite abelian groups.
	
	Then the subgroup $\Aut^0(G_\Gamma)$ of automorphisms generated by the set of all partial conjugations has finite index in $\Aut(G_\Gamma)$.
\end{prop}

We will focus throughout on the corresponding subgroup $\Out^0(G_\Gamma)$,
of finite index in $\Out(G_\Gamma)$.

\subsection{SILs, STILs, and FSILs}\label{sec:SILs STILS FSILS}

In the study of automorphisms of graph products, a feature of $\Gamma$ known as a SIL has become recognized as very significant. The reason for this is that it is precisely the condition on $\Gamma$ that allows for non-commuting partial conjugations.

\begin{defn}
    We say that $x_1, x_2, z$ form a \emph{SIL} (Separating Intersection of Links), denoted
    $(x_1,x_2\mids z)$, if $x_1$ and $x_2$ are not adjacent, and
    $\Gamma\setminus\left(\link{x_1}\cap\link{x_2}\right)$ has a
    component $C$ which contains $z$ and neither $x_1$ nor $x_2$.

    We will also talk of SILs $(x_1,x_2\mids Z)$ where $Z$ is a connected component of $\Gamma \setminus (\lk(x_1)\cap\lk(x_2))$ so that $(x_1,x_2\mids z)$ is a SIL whenever $z\in Z$.
\end{defn}

The following is proved in \cite[Section 4]{GPR_automorphisms} and describes precisely when partial conjugations do not commute.

\begin{lem}\label{lem:partial conj not commute}
     Suppose $x,y$ are distinct vertices. Two partial conjugations $\chi^x_C,\chi^y_D$ in $\Out(G_\Gamma)$ do not commute if and only if
     there is a SIL $(x,y\mids z)$ and either
     \begin{itemize}
        \item $z\in C=D$,
        \item $x\in D$ and $z\in C$,
        \item $y\in C$ and $z\in D$,
        \item $x\in D$ and $y\in C$.
     \end{itemize}
\end{lem}

The presence (or absence) of a SIL has been exploited by various authors (for example \cite{GPR_automorphisms,CRSV_no_SIL,GuirardelSale-vastness,Day_solvable}).
For example, Gutierrez, Piggott, and Ruane show the following structure theorem:

\begin{thm}[{\cite[Theorem 1.4]{GPR_automorphisms}}]
    Let $G_\Gamma$ be a graph product of finite cyclic groups of prime power order. Then $\Out(G_\Gamma)$ is finite if and only if $\Gamma$ does not contain a SIL.
\end{thm}

In applications for RAAGs, it is the free group generated by partial conjugations that is exploited.
For example, Guirardel and the first-named author use these free groups to show that when the SIL satisfies certain conditions, $\Out(A_\Gamma)$ is large \cite[Proposition 3.1]{GuirardelSale-vastness}.

When we look at RACGs, the non-commuting partial conjugations of a SIL generate the infinite dihedral group, which is virtually $\Z$,
so we cannot use SILs in the same way.
Instead we introduce two variations of a SIL that give us subgroups of $\Out(W_\Gamma)$ that are virtually non-abelian free.

\begin{defn}
    Let $\Gamma$ be a graph with vertices $x_1, x_2, x_3, z$.

\begin{enumerate}
\item We say that $x_1, x_2, x_3, z$ form a STIL (Separating Triple Intersection of Links), denoted $(x_1, x_2, x_3\mids z)$, if:
\begin{itemize}
	\item the full subgraph spanned by $\{x_1, x_2, x_3\}$ contains at most one edge, and
	\item  $\Gamma \setminus\left(\link{x_1}\cap\link{x_2}\cap\link{x_3}\right)$ has a component $C$ which contains $z$ and none of $x_1, x_2, x_3$.
\end{itemize}

\item We say that $x_1, x_2, x_3$ form an FSIL (Flexibly Separating Intersection of Links), if $(x_i, x_j\mids x_k)$ is a SIL for $\{i,j,k\}=\{1,2,3\}$.
\end{enumerate}
\end{defn}

As for SILs, for $Z\subseteq \Gamma$ we will say that $(x_1, x_2, x_3\mids Z)$ is a STIL if $(x_1, x_2, x_3\mids z)$ is a STIL for every vertex $z\in Z$.

We now collect two results that concern STILs and FSILs, particularly when you have overlapping SILs.

\begin{lem}\label{lem:no_overlap}
    Suppose that $\Gamma$ is connected and $(x_1, x_2\mids Z)$ and $(x_1, x_3\mids Z')$ are SILs.
    If $z\in Z\cap Z'$ then $(x_1,x_2,x_3\mids z)$ is a STIL.
    \end{lem}

\begin{proof} Suppose that $z\in Z \cap Z'$.

    Let $L=\lk(x_1)\cap\lk(x_2)\cap\lk(x_3)$. To show that $(x_1, x_2,x_3\mids z)$ is a STIL, we show there is no path from $z$ to $\{x_1,x_2,x_3\}$ which avoids $L$.

First, suppose that there there is a path from $z$ to $x_1$ avoiding $L$ and let $p$ be the shortest such path.
Since $(x_1, x_2\mids Z)$ is a SIL, and $z\in Z$,  there must be a vertex $v\in p\cap\lk(x_1)\cap\lk(x_2)$.
Since $p$ is the shortest path avoiding $L$ and $v$ is adjacent to $x_1$, $v$ must be the penultimate vertex on the path.
On the other hand, since $(x_1, x_3\mids Z')$ is a SIL, there must be a vertex $v'\in p\cap\lk(x_1)\cap\lk(x_3)$, and since $p$ avoids $L$ we must have that $v\neq v'$.
But, again, $p$ is the shortest path avoiding $L$ and $v'$ is adjacent to $x_1$, and thus $v'$ must be the penultimate vertex of $p$, a contradiction.
Thus, there is no path joining $z$ to $x_1$ avoiding $L$.

Suppose now that there was a path joining $z$ to $x_2$ avoiding $L$. Let $q$ be the shortest such path.
Then since $(x_1, x_2\mids Z)$ is a SIL, there is a vertex $v\in q\cap \lk(x_1)\cap\lk(x_2)$.
As above, we must have that $v$ is the penultimate vertex of $q$.
Let $e$ be the edge from $v$ to $x_1$, and let $q|_{[z,v]}$ be the subpath of $q$ from $z$ to $v$.
Then concatenating $q|_{[z,v]}$ with the edge $e$ gives a path from $z$ to $x_1$ that avoids $L$.
By the above argument, no such a path exists, giving a contradiction.
Similarly, there is no path from $z$ to $x_3$ avoiding $L$.
\end{proof}

\begin{lem}\label{lem:STILfind} Let $x_1,x_2,x_3\in \Gamma$ be such that there exist
    $y$ and $z$ with $(x_1, x_2\mids y)$ and $(x_1, x_3\mids z)$ both
    SILs. Then one of the following occurs:
    \begin{enumerate}
    	\item $\{x_1,x_2,x_3\}$ is an FSIL;
        \item there exists $w\in \Gamma$ so that $(x_1,x_2,x_3\mids w)$ is a
        STIL;
        \item $x_1, x_2, y$ are all in the same
        component of $\left(\Gamma\setminus\st(x_3)\right)\cup\{x_2\}$,
        and $x_1,x_3,z$ are all in the same component of $\left(\Gamma\setminus
        \st(x_2)\right)\cup\{x_3\}$;\label{item:y not 3 and z  not 2}
    \end{enumerate}\end{lem}

\begin{proof}
    Note first that by Lemma~\ref{lem:no_overlap}, if $y=z$ then $(x_1, x_2,x_3\mids y)$ is a STIL. We thus suppose that $y\neq
    z$. The remainder of the proof breaks into several cases.

    \case{1} If $y=x_3$ and $z=x_2$, then $\{x_1,x_2,x_3\}$ is an FSIL.
    
    To prove this, we just need to verify that $(x_2,x_3\mids x_1)$ is a SIL. Indeed, since $(x_1, x_2\mids x_3)$ is a SIL, we must have that $x_2$ and $x_3$ are non adjacent. Suppose that $(x_2, x_3 \mids x_1)$ is not a SIL, then there is a path from $x_1$ to, say, $x_2$ that bypasses $\lk(x_2)\cap\lk(x_3)$.
    Let $p$ be such a path of minimal length.
    Since $(x_1,x_3\mids x_2)$ is a SIL, $p$ must contain a vertex from $\lk(x_1)\cap\lk(x_3)$.
    If $v$ is the first vertex of $p$ in $\lk(x_1)\cap\lk(x_3)$, then we have a length two path, from $x_1$ to $v$ to $x_3$.
    Since $(x_1,x_2\mids x_3)$ is a SIL, $v\in \lk(x_1)\cap\lk(x_2)$, and thus $v\in \lk(x_2)\cap\lk(x_3)$, a contradiction. Thus, $(x_2, x_3\mids x_1)$ is a SIL.

   \case{2} If $y=x_3$ and $(x_1,x_3 \mids x_2)$ is not a SIL, then $(x_1,x_2,x_3\mids z)$ is a STIL.
   
    Since $(x_1,x_2\mids x_3)$ is a SIL, we know that the subgraph spanned by $\{x_1, x_2, x_3\}$ contains no edges and $\lk(x_1)\cap\lk(x_3) \subseteq \lk(x_1)\cap\lk(x_2)$,
    otherwise we can find a path from $x_3$ to $x_1$ avoiding $\lk(x_1)\cap\lk(x_2)$.
    Suppose there is a path from $z$ to $\{ x_1,x_2,x_3\}$ that bypasses $\lk(x_1)\cap\lk(x_2)\cap\lk(x_3)$, and let $q$ be such a path. 
    Since $\lk(x_1)\cap\lk(x_2)\cap\lk(x_3)=\lk(x_1)\cap\lk(x_3)$, and $(x_1,x_3\mids z)$ is a SIL, the terminus of $q$ must be $x_2$, and $q$ avoids $\lk(x_1)\cap\lk(x_3)$. 
    We then form a path $q'$ by concatenating $q$ with a path from $x_2$ to $\{x_1,x_3\}$ avoiding $\lk(x_1)\cap \lk(x_3)$, using the fact that $(x_1,x_3\mids x_2)$ is not a SIL. This is illustrated in Figure~\ref{fig:STILs}(\textsc{a}).
    This contradicts the fact that $(x_1,x_3 \mids z)$ is a SIL, and we thus conclude that $(x_1,x_2,x_3\mids z)$ is a STIL.

    \case{3} When neither $(x_1,x_2\mids x_3)$ nor $(x_1,x_3\mids x_2)$ are SILs.
    
    First note that the subgraph spanned by $\{x_1, x_2, x_3\}$ can contain at most one edge, between $x_2$ and $x_3$.
    Second, we claim that $x_2,x_3$ are in the same connected component of $\Gamma$ (and possibly also $x_1$).
    To see this, observe that there must be a path either from $x_2$ to $x_3$ or
    a path from $x_2$ to $x_1$ since $(x_1, x_3\mids x_2)$ is not a SIL. Further, there must also either be a
    path either from $x_3$ to $x_2$ or $x_3$ to $x_1$ since $(x_1,
    x_2\mids x_3)$ is not a SIL. In any of the cases above, there is
    a path from $x_2$ to $x_3$.
    
	Suppose now that $x_2, x_3$ are in a different connected component of $\Gamma$ to $x_1$. 
	Then $\lk(x_1)\cap\lk(x_i)=\emptyset$ for $i=2,3$, implying that $y$ is in a different connected component to $x_2, x_3$, as well as to $x_1$. Hence $(x_1,x_2,x_3\mids y)$ will form a STIL.  A similar argument applies to $z$.
    We can thus assume all of the vertices are in a single connected component of $\Gamma$.

    Next, we claim there can be no path from $y$ to $x_3$ avoiding
    $\lk(x_1)\cap\lk(x_2)$. Suppose that there is such a path $q$.
    Since $(x_1, x_2\mids x_3)$ is not a SIL, there is
    a path $q'$ from $x_3$ to $\{x_1, x_2\}$ avoiding
    $\lk(x_1)\cap\lk(x_2)$. The concatenation of $q$ and $q'$
    gives a path from $y$ to $\{x_1, x_2\}$ avoiding
    $\lk(x_1)\cap\lk(x_2)$, which is impossible since $(x_1, x_2\mids y)$ is a SIL. 
    This is illustrated in Figure~\ref{fig:STILs}(\textsc{b}). 

\begin{figure}
\centering \subfloat[The path $q$ can be extended to connect $z$ to
$\{x_1, x_3\}$.]{\label{subfig:STIL1}\begin{tikzpicture}[scale=0.6]
\tikzstyle{every node}=[shape=circle, color=black, fill=black]

\node[label=above:$x_1$] (1) at (4,6) {}; \node[label=below:$x_2$]
(2) at (4,0) {}; \node[label=below left:$x_3$] (3) at (0,3) {};
\node (11) at (2,4.5) {}; \node (12) at (6,4.5) {}; \node(131) at
(3,3.5) {}; \node(123) at (4,3) {}; \node (121) at (5, 2.5) {};
\node (21) at (2,1.5) {}; \node (22) at (6,1.5) {}; \node[label=
below right:$z$] (z) at (9,3) {};

\draw (1)--(11); \draw (1)--(12); \draw (1)--(131); \draw
(1)--(123); \draw (1)--(121); \draw (2)--(21); \draw (2)--(22);
\draw(2)--(121); \draw(2)--(123); \draw (2)--(131); \draw(3)--(131);
\draw(3)--(123); \draw[dashed](z) to[bend left=20] node[midway,
below, draw=none, fill=none]{$q$} (2);

\draw (4,3.75) ellipse (3cm and 1.75cm); \node[fill=none, draw=none]
at (6.75, 5.25){$\text{lk}(x_1)$}; \draw (4, 2.25) ellipse (3cm and
1.75cm);\node[fill=none, draw=none] at (1.25,
0.75){$\text{lk}(x_2)$};
\end{tikzpicture}}
\subfloat[The concatenation of $q$ and $q'$ gives a path from $y$ to
$x_1$]{\label{subfig:STIL2}\begin{tikzpicture}[scale=0.6]
\tikzstyle{every node}=[shape=circle, color=black, fill=black]
\node[label=above:$x_1$] (1) at (4,6) {}; \node[label=right:$x_2$]
(2) at (4,0) {}; \node[label=below left:$x_3$] (3) at (0,3)
{};\node[label= below right:$y$] (y) at (9,1.5) {};\node (11) at
(2.25,4.25) {}; \node (12) at (5.75,4.25) {}; \node(131) at (3,3.5)
{}; \node(123) at (4,3) {}; \node (21) at (2.25,1.75) {}; \node (22)
at (5.75,1.75) {};\node(31) at (1, 2.5) {}; \draw (1)--(11); \draw
(1)--(12); \draw (1)--(131); \draw (1)--(123);\draw (2)--(21); \draw
(2)--(22); \draw(2)--(123); \draw (2)--(131); \draw(3)--(131);
\draw(3)--(123); \draw(2)--(131); \draw(3)--(31); \draw[dashed](y)
to[bend left=75] node[midway, below, draw=none, fill=none]{$q$} (3);
\draw[dashed](3) to [bend left=30] node[midway, above, draw=none,
fill=none]{$q'$} (1);

\draw (4,3.75) ellipse (2.5cm and 1.75cm); \node[fill=none,
draw=none] at (6.75, 5.25){$\text{lk}(x_1)$}; \draw (4, 2.25)
ellipse (2.5cm and 1.75cm);\node[fill=none, draw=none] at (6.75,
0.75){$\text{lk}(x_2)$};

\end{tikzpicture}
} \caption{Figures illustrating Lemma~\ref{lem:STILfind}.}\label{fig:STILs}
\end{figure}
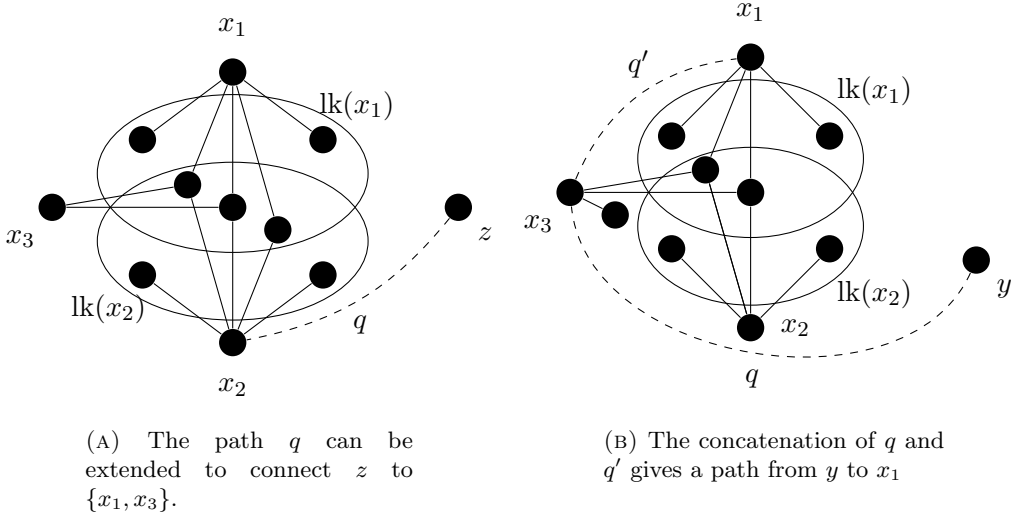

By the above, and the fact that $(x_1,x_2\mids y)$ is a SIL,
any path from $y$ to $\{x_1,x_2,x_3\}$ must pass through $\lk(x_1)\cap\lk(x_2)$. Suppose now that $(x_1, x_2, x_3\mids y)$ is not a STIL.
Then, there is a path from $y$ to $\lk(x_1)\cap\lk(x_2)$ which avoids $\lk(x_3)$, implying that $x_1,x_2,y$ are in the same connected component of $(\Gamma \setminus \st(x_3)) \cup \{x_2\}$.

Similarly, if $(x_1,x_2,x_3 \mids z)$ is not a STIL then 
$x_1,x_3,z$ are in the same connected component of $(\Gamma \setminus \st(x_2)) \cup \{x_3\}$.
\end{proof}

\section{The virtually abelian case}\label{sec:v ab}

We now focus on working with a RACG $W_\Gamma$.
The aim of this section is to prove the following.

\begin{thm}\label{thm:virtually abelian}
    Suppose $\Gamma$ has no STIL or FSIL. 
    
    Then $\Out(W_\Gamma)$ is virtually abelian.
\end{thm}

\subsection{Outline of the proof}

We split the proof up into two cases, one where $\Gamma$ is connected, and one where it is not.
Being disconnected and lacking any STIL or FSIL heavily restricts the structure of $\Gamma$, and hence of $W_\Gamma$ and $\Out(W_\Gamma)$.
Proposition~\ref{prop:disconnected_v_abelian} tells us in this situation that $\Out^0(W_\Gamma)$ is not just virtually abelian, but also a right-angled Coxeter group itself. The latter statement is implicit in \cite{Muhlherr:presentation, Tits:Coxter}, but we provide an independent proof here.

Having dealt with the case where $\Gamma$ is disconnected, we look to the case where the defining graph is connected.
The main tool in proving Theorem~\ref{thm:virtually abelian} then is the following assertion concerning the derived subgroup of $\Out^0(W_\Gamma)$.

\begin{prop}\label{prop:devived subgroup abelian}
    Suppose $\Gamma$ is connected and has no STIL or FSIL. 
    
    Then the derived subgroup $\Out^0(\cox\Gamma)'$ is abelian.
\end{prop}

Since $\Out^0(W_\Gamma)$ is finitely generated by finite order elements, it has finite abelianization, and so Theorem~\ref{thm:virtually abelian} follows from Proposition~\ref{prop:devived subgroup abelian}, when $\Gamma$ is connected.

The proof of Proposition~\ref{prop:devived subgroup abelian} has two steps.
First Lemma~\ref{lem:gen set derived subgroup} asserts that the derived subgroup is generated by commutators of partial conjugations.
In particular it is finitely generated.
The second step is to show that each of these commutators pairwise commute.
This is the subject of Lemma~\ref{lem:derived subgroup abelian}.

\subsection{Disconnected graphs}

We first wish to reduce to the case where $\Gamma$ is connected. So, as a special case, we deal with the case when $\Gamma$ is disconnected.

\begin{prop}\label{prop:disconnected_v_abelian}
  Suppose $\Gamma$ is disconnected and has no STIL or FSIL. 
  
  Then $\Out^0(W_\Gamma)$ is a virtually abelian right-angled Coxeter group.
\end{prop}

\begin{proof}
Suppose that $\Gamma$ is disconnected.
If there are at least three connected components of $\Gamma$, then picking a triple of vertices, each from distinct components, will produce an FSIL.
Thus, $\Gamma$ has two components, $\Gamma_1$ and $\Gamma_2$.
If one component contains three vertices so that at most one pair is adjacent,
then we can produce a STIL by taking this triple with a fourth vertex from the other component. Thus, for every triple of vertices in $\Gamma_i$, we must have at least two edges.
However, this implies that $\cox{\Gamma_i}$ is virtually abelian for $i=1,2$.

    For each $x\in \Gamma_i$, this means that $\Gamma_i\setminus\st(x_i)$ contains at most one vertex.
    Hence, $\Out^0(\cox{\Gamma_i})=1$.
    Let $j\neq i$, then for each $x\in \Gamma_i$, there is a partial conjugation $\chi^x_{\Gamma_j}\in\Out^0(\cox\Gamma)$.
    This will be the only partial conjugation in $\Out(W_\Gamma)$ with multiplier $x$ (and it may be inner).
    This allows us to identify each vertex of $\Gamma$ with a unique partial conjugation (though some may be inner),
    and the set of all partial conjugations in this identification generates $\Out^0(W_\Gamma)$.
    We use this identification and extend it to products of vertices to define a surjective map
    $$\iota\colon \cox{\Gamma_1}\times \cox{\Gamma_2}\to \Out^0(\cox\Gamma).$$
    We first check the relations are preserved so we see $\iota$ is a homomorphism.

    If $x,y\in\Gamma_i$, then $(x,y\mids \Gamma_j)$ is a SIL if and only $x$ and $y$ are not joined by an edge,
    and thus by Lemma~\ref{lem:partial conj not commute}, $\chi^x_{\Gamma_j}$ and $\chi^y_{\Gamma_j}$ commute if and only if $x$ and $y$ are adjacent.
    If $x\in \Gamma_1$ and $z\in\Gamma_2$, then $\chi^x_{\Gamma_2}$ commutes with $\chi^z_{\Gamma_1}$ by Lemma~\ref{lem:partial conj not commute}, since $x$ and $z$ cannot form a SIL (since $\Gamma$ has only two connected components).
    Thus $\iota$ is a homomorphism.

    We now determine the kernel of $\iota$.
    Take $w_1\in \cox{\Gamma_1}$ and $w_2\in \cox{\Gamma_2}$ with $\iota((w_1,w_2))=1$.
    We abuse notation slightly and consider the partial conjugation $\chi^{w_i}_{\Gamma_i}$, conjugating every vertex in $\Gamma_i$ by $w_i$.
    In $\Out^0(\cox\Gamma)$, we have  $\chi^{w_i}_{\Gamma_j} \chi^{w_i}_{\Gamma_i} = 1$,
    so
    $$1 = \iota((w_1, w_2))=\chi^{w_1}_{\Gamma_2}\chi^{w_2}_{\Gamma_1}=(\chi^{w_1}_{\Gamma_1})\m (\chi^{w_2}_{\Gamma_2})\m=\left(\chi^{w_2}_{\Gamma_2}\chi^{w_1}_{\Gamma_1}\right)\m.$$
    This will be trivial exactly when $w_1\in Z(\cox{\Gamma_1})$ and $w_2\in Z(\cox{\Gamma_2})$.
    Thus, $\ker(\iota)=Z(\cox{\Gamma_1})\times Z(\cox{\Gamma_2})$, and so
    $$\Out^0(\cox\Gamma)\cong \cox{\Gamma_1}/Z(\cox{\Gamma_1})\times \cox{\Gamma_2}/Z(\cox{\Gamma_2}).$$

    Furthermore, $Z(\cox{\Gamma_i})$ is the subgroup of $\cox{\Gamma_i}$ generated by those vertices which are adjacent to every vertex of $\Gamma_i$, and $\cox{\Gamma_i}/Z(\cox{\Gamma_i})$ is isomorphic to the right-angled Coxeter group with defining graph $\Gamma_i\setminus K_i$, where $K_i$ is the clique of vertices generating $Z(\cox{\Gamma_i})$. Thus, $\Out^0(\cox\Gamma)$ is a right-angled Coxeter group whose defining graph is the join of $\Gamma_1\setminus K_1$ and $\Gamma_2\setminus K_2$.
    
    Finally, since $W_{\Gamma_1}$ and $W_{\Gamma_2}$ are virtually abelian, $\Out^0(W_\Gamma)$ is virtually abelian too.
\end{proof}

\subsection{Generating the derived subgroup}

Given a group with generating set $X$, the derived subgroup is generated by all conjugates of commutators of elements in $X$.
By studying the behaviour of conjugates of commutators of partial conjugations, we show the following.

\begin{lem}\label{lem:gen set derived subgroup}
    Suppose $\Gamma$ is connected and has no STIL or FSIL. 
    
    Then $\Out^0(\cox\Gamma)'$ is generated by commutators of partial conjugations.
\end{lem}

To prove this, we show that given three partial conjugations $\chi_1,\chi_2,\chi_3$, the element given by $\chi_1[\chi_2,\chi_3]\chi_1$ is a product of commutators of partial conjugations.
In fact, in most cases the conjugate $\chi_1[\chi_2,\chi_3]\chi_1$ is equal to $[\chi_2,\chi_3]$ or its inverse. The situation is divided into two cases, according to how many unique multipliers are involved in the three partial conjugations.
Note that, necessarily, $\chi_2$ and $\chi_3$ must have distinct multipliers, and their multipliers must form a SIL, otherwise the result is trivial.

We begin with the case when two partial conjugations share the same multiplier,
since the technical details of this situation are needed in the case when all multipliers are distinct.

\begin{lem}\label{lem:conjugating commutators with two vertices involved}
   Let $x_1,x_2$ be distinct vertices of $\Gamma$, and let $\chi=\chi^{x_1}_C$ be a partial conjugation, and
    $\theta_1=\chi^{x_2}_{C_1},\ldots,\theta_r = \chi^{x_2}_{C_r}$ be the complete list of partial conjugations with multiplier $x_2$ (where each $C_i$ is a connected component of $\Gamma\setminus\st(x_2)$).
    Order these so that $x_1\in C_1$.
\begin{enumerate}[(i)]
	\item  \label{item:conj of com is com inv} 
	If either
	\begin{enumerate}
		\item $i=j$,
		\item $i>j=1$, $(x_1,x_2\mids C_i)$ is a SIL, and either $C_i=C$ or $x_2\in C$,
		\item $i=1<j$, $C_j=C$, and $(x_1,x_2\mids C)$ is a SIL,
	\end{enumerate}
then $$\theta_j [\chi,\theta_i]\theta_j = [\theta_i,\chi] = [\chi,\theta_i]\m.$$
	\item \label{item:conj of com is complicated}
	If $i=1<j$, $(x_1,x_2 \mids C_j)$ is a SIL, and $x_2\in C$, then
	$$\theta_j [\chi,\theta_i]\theta_j =  [\chi, \theta_2]\ldots[\chi, \theta_{j-1}][\theta_j,\chi][\chi,\theta_{j+1}]\ldots [\chi,\theta_r].$$
	\item  \label{item:conj of com is com} Otherwise $\theta_j [\chi,\theta_i]\theta_j = [\chi,\theta_i]$.
\end{enumerate}
In particular $\theta_j [\chi,\theta_i]\theta_j$ is a product of commutators of partial conjugations, for any  choice of $i,j$.
\end{lem}

\begin{rem}
    Lemma~\ref{lem:conjugating commutators with two vertices involved} implies that
    $\theta_j [\chi,\theta_i]\theta_j = [\chi,\theta_i]^{\pm 1}$ in all cases except when the elements in the commutator act on each other's multiplier. Notably, there are no assumptions on $\Gamma$ here.
\end{rem}

\begin{proof}[Proof of Lemma~\ref{lem:conjugating commutators with two vertices involved}]
	Let $C_i$ denote the support of $\theta_i$.
	
	First observe that if $i=j$ then we get $\theta_j [\chi,\theta_i]\theta_j=[\theta_i,\chi]$ as in \eqref{item:conj of com is com inv}(a). Thus we assume $i\neq j$.
	
	The equality $\theta_j [\chi,\theta_i]\theta_j = [\chi,\theta_i]$ in \eqref{item:conj of com is com} is obvious whenever either $[\chi, \theta_i]=1$ or $[\theta_j,\chi]=1$. 
	By Lemma~\ref{lem:partial conj not commute},
	to avoid this, we require a SIL $(x_1,x_2 \mids Z)$ with either
	\begin{enumerate}[(I)]
		\item $Z=C=C_i$,
		\item $x_1\in C_i$ (so $i=1$) and $Z=C$,
		\item $Z=C_i$ and $x_2\in C$,
		\item $x_1\in C_i$ (so $i=1$) and $x_2\in C$;
	\end{enumerate}
and either
	\begin{enumerate}[(A)]
		\item $Z=C=C_j$,
		\item $x_1\in C_j$ (so $j=1$) and $Z=C$,
		\item $Z=C_j$ and $x_2\in C$,
		\item $x_1\in C_j$ (so $j=1$) and $x_2\in C$.
	\end{enumerate}
Most combinations are not possible. For example (II) and (C) is not possible since together they imply $x_2\in C_j$, which is not possible since $C_j$ is a connected component of $\Gamma\setminus \st(x_2)$.
Or (IV) and (B) is not possible since it implies $x_1\in C_i\cap C_j$, which is not possible since $C_i\cap C_j = \emptyset$.

There are only four possible combinations: (I) and (B), (II) and (A), (III) and (D), (IV) and (C).
Note that each of these imply either $i=1$ or $j=1$. So we have:
	\begin{equation}\label{eq:i neq j commuting}
	\theta_j[\chi,\theta_i]\theta_j = [\chi,\theta_i] \textrm{  if $i\neq j$ and $i,j>1$.}
	\end{equation}

\case{1} (I) and (B).

Here we have $C=C_i$, with $(x_1,x_2\mids C_i)$ forming a SIL, and $j=1$.
Since $\chi$ and $\theta_i$ both have support equal to $C$, and $x_1,x_2\notin C$, 
we get $[\chi , \theta_i] = \chi^{[x_2,x_1]}_C$,
meaning that any $v\in C$ is sent to $[x_2,x_1]v[x_1,x_2]$ and all other vertices of $\Gamma$ are fixed.
For $u\in C_1$ and $v\in C$ we have
\begin{align*}
\theta_1 [\chi ,\theta_i] \theta_1 (u) & = \theta_1\chi^{[x_2,x_1]}_C (x_2 u x_2)  = \theta_1 (x_2 u x_2) = u \\
\theta_1 [\chi ,\theta_i] \theta_1 (v) & = \theta_1\chi^{[x_2,x_1]}_C(v) = \theta_1 ( x_2x_1x_2x_1 v x_1x_2x_1x_2 ) = [x_1,x_2] v [x_2,x_1]
\end{align*}
with all other vertices fixed.
This verifies that $\theta_1 [\chi,\theta_i]\theta_1 = \chi^{[x_1,x_2]}_C = [\chi,\theta_i]\m$, proving half of \eqref{item:conj of com is com inv}(b) in the statement of the Lemma.

\case{2} (II) and (A).

We claim that $[\chi,\theta_1] = \chi^{[x_2,x_1]}_C$.
Note that, since $(x_1,x_2\mids C)$ is a SIL, we get that $C\cap C_1 = \emptyset$.
We have $x_1\in C_1$ and $x_2$ fixed by both $\chi$ and $\theta_1$.
Then for $u\in C$ and $v\in C_1$ we have
\begin{align*}
\chi\theta_1\chi\theta_1(u) & = \chi\theta_1\chi(u)  = \chi\theta_1(x_1ux_1)  = \chi(x_2x_1x_2ux_2x_1x_2)  = [x_2,x_1]u[x_1,x_2]\\
\chi\theta_1\chi\theta_1(v) & = \chi\theta_1\chi(x_2vx_2)  =\chi\theta_1(x_2vx_2)  =\chi(v) =v,
\end{align*}
and all other vertices are fixed. This proves the claim.

Then $\theta_j[\chi,\theta_1]\theta_j = \chi^{x_2}_{C}\chi^{[x_2,x_1]}_C\chi^{x_2}_C$. Since neither $x_1$ nor $x_2$ is in $C$, this implies that $\theta_j[\chi,\theta_1]\theta_j = \chi^{[x_1,x_2]}_C = [\chi,\theta_1]\m$, giving \eqref{item:conj of com is com inv}(c).

\case{3} (III) and (D).

Here we have $(x_1,x_2\mids C_i)$ forming a SIL, so $i>1$, along with $x_2\in C$ and $j=1$.
First, the commutator $[\chi,\theta_i]$ is equal to $\chi^{[x_2,x_1]}_{C_i}$: for $u\in C_i$ and $v\in C$ we have
\begin{align*}
\chi\theta_i\chi\theta_i(u) & = \chi\theta_i\chi(x_2ux_2)  = \chi\theta_i(x_1x_2x_1ux_1x_2x_1)  = \chi(x_1x_2x_1x_2ux_2x_1x_2x_1)  = [x_2,x_1]u[x_1,x_2]\\
\chi\theta_i\chi\theta_i(v) & = \chi\theta_i\chi(v)  =\chi\theta_i(x_1vx_1)  =\chi(x_1vx_1) = v,
\end{align*}
while all other vertices are fixed.

For $u\in C_i$ and $w\in C_1$ we have
\begin{align*}
\theta_1 [\chi ,\theta_i] \theta_1 (u) & = \theta_1\chi^{[x_2,x_1]}_{C_i} ( u )  = \theta_1 (x_2x_1x_2x_1 u x_1x_2x_1x_2) = [x_1,x_2]u[x_2,x_1] \\
\theta_1 [\chi ,\theta_i] \theta_1 (w) & = \theta_1\chi^{[x_2,x_1]}_{C_i}(x_2wx_2) = \theta_1 ( x_2wx_2 ) = w,
\end{align*}
and fixing all other vertices. This tells us that $\theta_1[\chi,\theta_i]\theta_1 = \chi^{[x_1,x_2]}_{C_i} = [\chi,\theta_i]\m$, and proves the second half of \eqref{item:conj of com is com inv}(b)

\case{4} (IV) and (C).

In this case we have $i=1$, $(x_1,x_2\mids C_j)$ forming a SIL, and $x_2\in C$.
Since $\theta_1=\theta_2\cdots\theta_r$ in $\Out(W_\Gamma)$, we get
\begin{align*}
[\chi ,\theta_1]
& = [\chi , \theta_2\cdots \theta_r] \\
& = \chi\theta_2\cdots \theta_r \chi \theta_r \cdots \theta_2 \\
& = \chi \theta_2 \cdots \theta_{r-1}\chi [\chi,\theta_r] \theta_{r-1} \cdots \theta_2 \\
& = [\chi , \theta_2\cdots\theta_{r-1}][\chi,\theta_r].
\end{align*}
The last line follows from the preceding since $[\chi,\theta_r]$ commutes with each $\theta_k$ for $1<k<r$ by equation~\eqref{eq:i neq j commuting}.
A simple induction then gives
\begin{equation} \label{eq:rewriting a commutator} 
[\chi,\theta_1] = [\chi,\theta_2]\cdots[\chi,\theta_r].
\end{equation}
Thus, when conjugating by $\theta_j$ for $j>1$, we can apply equation~\eqref{eq:i neq j commuting} and obtain \eqref{item:conj of com is complicated}, since each commutator $[\chi, \theta_k]$ commutes with $\theta_j$, for $1<k$, except when $k=j$, in which case we get $\theta_j[\chi,\theta_j]\theta_j = [\theta_j,\chi]$.
\end{proof}

From the proof of Lemma~\ref{lem:conjugating commutators with two vertices involved}, specifically equation~\eqref{eq:rewriting a commutator}, we get the following.

\begin{lem}\label{lem:rewriting commutator}
	Let $\chi_1 = \chi^{x_1}_{C_1}$ and $\chi_2 = \chi^{x_2}_{C_2}$ be partial conjugations with distinct multipliers.
	Let $\chi_{2,1} = \chi_2, \chi_{2,2},\ldots, \chi_{2,r}$ be a complete list of partial conjugations with multiplier $x_2$.
	
	If $x_1 \in C_2$ then
	$$[\chi_1,\chi_2] = [\chi_1, \chi_{2,2}]\cdots[\chi,\chi_{2,r}].$$
\end{lem}

We then use this in dealing with the case with three distinct multipliers.

\begin{lem}\label{lem:conjugating commutators with three vertices involved}
    Suppose $\Gamma$ is connected and has no STIL or FSIL.
    Let $\chi_i=\chi^{x_i}_{C_i}$, for $i=1,2,3$ be partial conjugations in $\Out(W_\Gamma)$ with distinct multipliers.
    
    Then  $$\chi_3 [\chi_1,\chi_2]\chi_3 = [\chi_1,\chi_2].$$
\end{lem}
\begin{rem}\label{rem:conjugate of commutator usuall doesnt do much}
    In the absence of a STIL or FSIL, whenever $\Gamma$ is connected, Lemmas~\ref{lem:conjugating commutators with two vertices involved} and \ref{lem:conjugating commutators with three vertices involved} imply that
    $\chi_3 [\chi_1,\chi_2]\chi_3 = [\chi_1,\chi_2]^{\pm 1}$ unless $\chi_1$ and $\chi_2$ act non-trivially on each other's multiplier, $\chi_3$ shares a multiplier with one of $\chi_1$ or $\chi_2$, and the multipliers of $\chi_1$ and $\chi_2$ form a SIL with the support of $\chi_3$.
\end{rem}

\begin{proof}

    \begin{figure}[h!]
        \begin{tikzpicture}
\tikzstyle{every node}=[shape=circle, color=black, fill=black]
\node[label=above:$x_1$] (1) at (5,3) {}; \node[label=right:$x_2$]
(2) at (3,6) {}; \node[label=left:$x_3$] (3) at (7,6) {}; \node (12)
at (3.5,3.75) {}; \node (13) at (6.5, 3.75) {}; \node[draw=none,
fill=none] (121) at (2,3) {}; \node[draw=none, fill=none] (131) at
(8,3) {}; \node[draw=none, fill=none] (21) at (2.5, 7) {};
\node[draw=none, fill=none] (22) at (3.5,7) {}; \node[draw=none,
fill=none] (31) at (6.5,7) {}; \node[draw=none, fill=none] (32) at
(7.5,7) {}; \node[draw=none, fill=none] (11) at (4, 1.5) {};
\node[draw=none, fill=none] (14) at (5,1) {}; \node[draw=none,
fill=none] (15) at (6, 1.5) {}; \node[draw=none, fill=none, inner
sep=0pt] (130) at (7.4, 3.3) {}; \node[draw=none, fill=none, inner
sep=0pt] (120) at (2.6, 3.3) {}; \node[draw=none, fill=none, inner
sep=0pt] (201) at (2.75, 6.5) {}; \node[draw=none, fill=none, inner
sep=0pt] (202) at (3.25, 6.5) {}; \node[draw=none, fill=none, inner
sep=0pt] (301) at (6.75, 6.5) {}; \node[draw=none, fill=none, inner
sep=0pt] (302) at (7.25,6.5) {}; \node[draw=none, fill=none, inner
sep=0pt] (101) at (13/3,2) {}; \node[draw=none, fill=none, inner
sep=0pt] (104) at (5,1.75) {}; \node[draw=none, fill=none, inner
sep=0pt] (105) at (17/3,2) {};

\draw (1)--(101); \draw[dashed] (101)--(11); \draw (1)--(12); \draw
(1)--(13); \draw (1)--(104); \draw[dashed] (104)--(14); \draw
(1)--(105); \draw[dashed](105)--(15); \draw (2)--(12); \draw
(2)--(201); \draw(2)--(202); \draw[dashed] (201)--(21);
\draw[dashed](202)--(22); \draw(3)--(13); \draw(3)--(301); \draw
(3)--(302); \draw[dashed] (301)--(31); \draw[dashed] (302)--(32);
\draw (12) -- (120); \draw[dashed] (120) -- (121); \draw
(13)--(130); \draw[dashed] (130)--(131);

\node[draw=none, fill=none] at (5,7) {{\Large $C_1$}};
\node[draw=none, fill=none] at (1.5,2.5) {{\Large $C_2=Z$}};
\node[draw=none, fill=none] at (8.25,2.5) {{\Large $Z'$}};
\draw[dashed, color=gray] (2,6.5) to[bend right=75] (8,6.5);
\draw[dashed, color=gray] (1.25, 4.3)  to[bend left=65] (2.5,1.5);
\draw[dashed, color=gray] (8.75, 4.3)  to[bend right=65] (7.5,1.5);
\end{tikzpicture}
        \caption{Without loss of generality, we can consider the case when $x_2$ and $x_3$ are both in $C_1$, and $(x_1,x_2\mids C_2)$ is a SIL.}\label{fig:three distinct multipliers}
    \end{figure}
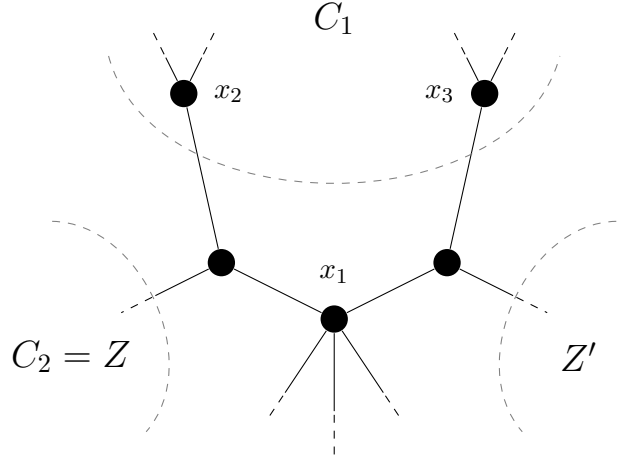

	The result is trivial if $[\chi_1,\chi_2]=1$, so we assume otherwise.
    In order for $[\chi_1,\chi_2]\neq 1$ we require a SIL $(x_1,x_2 \mids Z)$ and either $x_2\in C_1$ or $C_1=Z$.
    Furthermore, without loss of generality we may assume that $\chi_1$ and $\chi_3$ also do not commute, since if $\chi_3$ commutes with both $\chi_1$ and $\chi_2$ then the result is again immediate. 
    For $[\chi_1,\chi_3]\neq 1$ we need a SIL $(x_1,\chi_2\mids Z')$ and either $x_3\in C_1$ or $C_1=Z'$.
    By Lemma~\ref{lem:no_overlap}, the absence of a STIL in $\Gamma$ implies $Z\cap Z'=\emptyset$.
    From Lemma~\ref{lem:STILfind}, lacking a STIL or FSIL means we must have $\{x_1,x_2\}\cup Z$ contained in one connected component of $(\Gamma\setminus\st(x_3))\cup\{x_2\}$. 
    Similarly, $\{x_1,x_3\}\cup Z'$ lies on one connected component of $(\Gamma\setminus\st(x_2))\cup\{x_3\}$. 
    In particular, $x_3 \notin Z$ and $x_2\notin Z'$.
    This means we must have $x_2,x_3\in C_1$, since any of the other possible combinations lead to a contradiction.
    
    Since $[\chi_1,\chi_2]\neq 1$, we will have either $C_2 = Z$ or $x_1\in C_2$.
    Similarly $[\chi_1,\chi_3]\neq 1$ implies either $C_3=Z'$ or $x_1\in C_3$.
    Assume that $x_1\in C_2$.
    Let $\chi_2=\chi_{2,1},\chi_{2,2},\ldots,\chi_{2,r}$ be the list of partial conjugations with multiplier $x_2$.
    Then $[\chi_1,\chi_2] = [\chi_1,\chi_{2,2}]\cdots [\chi_1,\chi_{2,r}]$ by Lemma~\ref{lem:rewriting commutator}.
    We are therefore reduced to proving the lemma when $x_1\notin C_2$.

	Thus, we may assume $C_2=Z$, and either $x_1\in C_3$ or $C_3=Z'$, as depicted in Figure~\ref{fig:three distinct multipliers}.
    Direct calculation shows that $[\chi_1,\chi_2]$ is equal to $\chi^{[x_2,x_1]}_{C_2}$, sending each $v\in C_2$ to $[x_2,x_1]v[x_1,x_2]$, and fixing all other generators.
    
    If $C_3=Z'$ then we can see $[\chi_3,[\chi_1,\chi_2]]=1$ since $\chi_3$ and $[\chi_1,\chi_2]$ have disjoint supports in $\Gamma$ and both multipliers are fixed.

    We are left with the case when $x_1\in C_3$.
    By Lemma~\ref{lem:STILfind}(3), $\{x_1\}\cup C_2 \subseteq C_3$ and $x_2\in C_3$ too, else $[x_2,x_3]=1$.
    Direct calculations, left to the reader to verify, yield that $[\chi_3,[\chi_1,\chi_2]]=1$, and the proof of the lemma is complete.
    \end{proof}

\begin{proof}[Proof of Lemma~\ref{lem:gen set derived subgroup}]
    The derived subgroup $\Out^0(W_\Gamma)'$ is generated by all conjugates of commutators of partial conjugations.
    Let $\chi_1,\chi_2,\chi_3$ be three partial conjugations.
   Lemmas~\ref{lem:conjugating commutators with two vertices involved} ~and~\ref{lem:conjugating commutators with three vertices involved} imply that $\chi_3[\chi_1,\chi_2]\chi_3$ is a product commutators of partial conjugations. 
   By induction, for any $g\in \Out^0(W_\Gamma)$, $g[\chi_1,\chi_2]g^{-1}$ is a product of commutator of partial conjugations.
\end{proof}

\subsection{Commuting commutators}

To complete the proof of Proposition~\ref{prop:devived subgroup abelian}, in light of Lemma~\ref{lem:gen set derived subgroup} we need only check that commutators of partial conjugations commute with each other:

\begin{lem}\label{lem:derived subgroup abelian}
    Suppose $\Gamma$ is connected and has no STIL or FSIL. Then, for any four partial conjugations $\chi_1,\chi_2,\chi_3,\chi_4$ in $\Out(W_\Gamma)$, we have
    $$[\chi_1,\chi_2][\chi_3,\chi_4] = [\chi_3,\chi_4][\chi_1,\chi_2].$$
\end{lem}

\begin{proof}
Let $\chi_i=\chi^{x_i}_{C_i}$ for $i=1,2,3,4$.
We may assume that $x_1\neq x_2$ and $x_3\neq x_4$, so both commutators are non-trivial.

First note that if all vertices $x_1,x_2,x_3,x_4$ are distinct then Lemma~\ref{lem:conjugating commutators with three vertices involved}
tells us that the commutators commute.
So we may assume, say, that $x_4=x_1$.

Next assume that $x_2 \neq x_3$.
Then by Lemma~\ref{lem:conjugating commutators with three vertices involved},
$\chi_3[\theta,\chi_2]\chi_3 = [\theta,\chi_2]$,
where $\theta$ is any partial conjugation with multiplier $x_1$, since $x_1=x_4\neq x_3$.
Since Lemma~\ref{lem:conjugating commutators with two vertices involved} tells us that
$\chi_4[\chi_1,\chi_2]\chi_4$ is a product of commutators of the form $[\theta,\chi_2]^{\pm 1}$,
we see that it commutes with $\chi_3$, and hence
        $$[\chi_4,\chi_3][\chi_1,\chi_2][\chi_3,\chi_4] = \chi_4\chi_3\left(\chi_4[\chi_1,\chi_2]\chi_4\right)\chi_3\chi_4=\chi_4^2 [\chi_1,\chi_2]\chi_4^2=[\chi_1,\chi_2].$$

We finish by assuming that $x_2=x_3$.

    Let $\theta_1,\ldots,\theta_r$ be all the partial conjugations with multiplier $x_1$, and let $\phi_1,\ldots,\phi_s$ be the partial conjugations with multiplier $x_2$.

    We claim $[\phi_i,\theta_j]$ commutes with $[\phi_k,\theta_l]$ for $1\leq i,k\leq r$ and $1\leq j,l\leq s$.

    In order for the claim to be non-trivial, we need a SIL of the form $(x_1,x_2\mids z)$ for some vertex $z$.
    Let $C_i$ be the support of $\theta_i$ and $D_i$ the support of $\phi_i$.
    Reorder the partial commutators if necessary so that $x_1\in D_1$, $x_2\in C_1$, and $C_2=D_2,\ldots, C_m=D_m$, where $(x,y\mids C_i)$, for $i=2,\ldots ,m$, exhausts the list of SILs involving $x_1,x_2$.
	If $i=j=k=l=1$, the result is trivial. So first suppose that $(k,l)\neq (1,1)$.
    We apply Lemma~\ref{lem:conjugating commutators with two vertices involved} twice to show that:
            $$
            \phi_i\theta_j[\phi_k,\theta_l]\theta_j\phi_i = [\phi_k,\theta_l]^{\pm 1}.
            $$
    This implies that $[\phi_i,\theta_j]$ commutes with $[\phi_k,\theta_l]$ when $(k,l)\neq(1,1)$. If $(k,l)=(1,1)$, then $(i,j)\neq (1,1)$, so we may swap the roles of $\{i,j\}$ and $\{k,l\}$ to obtain the same result, completing the proof.
\end{proof}

\section{The large case}\label{sec:large case}\label{sec:large}

In this section we prove the other half of the dichotomy.

\begin{thm}\label{thm:STIL FSIL implies large}
    Suppose $\Gamma$ contains a STIL or an FSIL. 
    
    Then $\Out(W_\Gamma)$ is large.
\end{thm}

Our main tool are factor maps into the automorphism groups of subgraphs of $\Gamma$.
These have been previously established for $\Aut(W_\Gamma)$ \cite[Theorem 3.9]{GPR_automorphisms},
and they allow us to focus our attention on a chosen few partial conjugations.

Given a full subgraph $\Gamma'$ of $\Gamma$, let $\kappa$ be the projection map
$\kappa \colon W_\Gamma \to W_{\Gamma'}$ obtained by sending all vertices in $\Gamma\setminus\Gamma'$ to $1$.
We use this to define the \emph{factor map} for $\Gamma'$:
$$f \colon \Aut^0(W_\Gamma) \to \Aut^0(W_{\Gamma'}).$$
For $\phi \in \Aut^0(W_\Gamma)$ and $g\in W_\Gamma$, define $f(\phi)$ to send $\kappa(g)$ to  $\kappa(\phi(g))$.
Since the kernel of $\kappa$ is the normal subgroup generated by $\Gamma\setminus\Gamma'$,
$\phi(\ker(\kappa))=\ker(\kappa)$, and so this gives a well-defined homomorphism.

Further, $f$ takes inner automorphisms to inner automorphisms and so $f$ descends to a homomorphism (which, abusing notation, we also call $f$):
$$f\colon \Out^0(W_\Gamma)\to\Out^0(W_\Gamma').$$

Let $\Gamma_3$ be the discrete graph on 3 vertices, $\Gamma_4$ be
the discrete graph on 4 vertices, and $\Gamma_{3,1}$ be the graph on
4 vertices with one edge.
If $\Gamma$ contains an FSIL, then the factor map $f$ has image in $\outo{\cox{\Gamma_3}}$.
Meanwhile, if $\Gamma$ contains a STIL, then the homomorphism $f$ maps $\Out^0(W_\Gamma)$ into either
$\outo{\cox{\Gamma_4}}$ or $\outo{\cox{\Gamma_{3,1}}}$.

To analyze these cases, we require the following theorem of Collins.

\begin{thm}[{\cite[Corollary 3.2]{Collins:autofreefinite}}] \label{thm:virtfree} 
    Let $G_1, G_2, G_3$ be finite groups. The virtual cohomological dimension of $\Out\left(G_1\ast G_2\ast
G_3\right)$ is $1$.

    In particular, $\Out(G_1 \ast G_2 \ast G_3)$ is virtually free.
\end{thm}

\begin{rem}\label{rem:collins not vZ}
    The group $\Out(G_1 \ast G_2 \ast G_3)$ is not virtually cyclic.
    Indeed, we can define automorphisms of the free product by taking $g_i$ to be a non-trivial element in $G_i$, and defining $\chi_i$ to fix $G_i$ and $G_{i-1}$, and to conjugate $G_{i+1}$ by $g_i$ (we take indices modulo 3).
    Then $\langle \chi_1,\chi_2,\chi_3\rangle$ will generate a subgroup of $\Out(G_1\ast G_2 \ast G_3)$ that is isomorphic to $\Z_{n_1}\ast \Z_{n_2} \ast \Z_{n_3}$, where $n_i$ is the order of $g_i$.
\end{rem}

In light of Theorem~\ref{thm:virtfree} and Remark~\ref{rem:collins not vZ},
by taking the factor map to the vertices of an FSIL,
we get a map from $\Out^0(W_\Gamma)$ to a virtually non-abelian free group.
The flexibility of the SIL ensures we have sufficient partial conjugations in the image of the factor map to make sure it is not virtually cyclic.
This gives the following.

\begin{prop}\label{prop:FSIL_large} Suppose that $\Gamma$ contains an
FSIL. 

Then $\Out(W_\Gamma)$ is large.
\end{prop}

\begin{proof} Let $\{x_1,x_2,x_3\}$ be vertices of $\Gamma$ defining an FSIL. Then they span a full copy of $\Gamma_3\subseteq \Gamma$. Let $f\colon
\outo{\cox\Gamma}\to\outo{\cox{\Gamma_3}}$ be the corresponding
factor map. Since $\{x_1, x_2,x_3\}$ form an FSIL, the image of
this homomorphism must contain $\chi^{x_i}_{x_j}$ for all $i\neq
j\in\{1,2,3\}$. However, this generates $\outo{\cox{\Gamma_3}}$,
which is virtually non-abelian free by Theorem~\ref{thm:virtfree} and Remark~\ref{rem:collins not vZ}.
Thus, $\outoW$ surjects onto a virtually free group, and so is
large.
\end{proof}

We now consider the case where $\Gamma$ contains a STIL $(x_1, x_2,
x_3\mids x_4)$. This breaks into two cases, depending on whether the
full subgraph spanned by $\{x_1, x_2, x_3\}$ contains an edge.
When it does contain an edge, the factor map has image contained in $\Out^0(W_{\Gamma_{3,1}})$, which is virtually free by Theorem~\ref{thm:virtfree}.
As in the FSIL case, we just need to check the image is not virtually cyclic.

\begin{prop}\label{prop:STIL_edge_large}
Suppose that $\Gamma$ contains a STIL $(x_1,x_2,x_3\mids x_4)$ and
that $x_1$ and $x_2$ are joined by an edge in $\Gamma$. 

Then $\Out(W_\Gamma)$ is large.
\end{prop}

\begin{proof}
Let $f\colon\outoW\to\outo{\cox{\Gamma_{3,1}}}$
be the factor map induced by the full subgraph on vertices $\{x_1, x_2, x_3, x_4\}$.
Theorem~\ref{thm:virtfree} tells us $\outo{\cox{\Gamma_{3,1}}}$ is virtually
free, thus it is sufficient to prove that the image of the factor map
is not virtually cyclic.

The image contains the partial conjugations
$\chi^{x_i}_{x_4}$ for $1\le i\le 3$.
Consider a word:
$$\Theta=\chi^{x_{i_1}}_{x_4}\chi^{x_{i_2}}_{x_4}\ldots\
\chi^{x_{i_\ell}}_{x_4}.$$
Then $\Theta(x_i)=x_i$ for $i=1,2,3$, and
$\Theta(x_4)=x_{i_\ell}\ldots x_{i_1} x_4 x_{i_1}\ldots x_{i_\ell}$.
This is a non-trivial automorphism so long as $x_{i_1}\ldots
x_{i_\ell}\neq 1$ in $\cox{\{x_1,x_2,x_3\}}$.
Thus, we get an
embedding $\cox{\{x_1,x_2,x_3\}}\cong \left(\Z_2\times\Z_2\right)\ast\Z_2\hookrightarrow
\Im(f)$ and hence the image of $f$ is not virtually cyclic.
\end{proof}

When the full subgraph spanned by the STIL contains no edges,
we need to consider the partial conjugations present in the image of the factor map.
There are various possibilities regarding whether or not certain partial conjugations are present.
The following lemmas reduce the number of cases necessary to consider by showing that in certain situations we will have an FSIL,
with the result then following from by Proposition~\ref{prop:FSIL_large}.

The first lemma says that we can always consider cases where the image of $f$ does not contain any of the partial conjugations with multiplier $x_4$.

\begin{lem}\label{lem:PC with multiplier 4 gives FSIL}
    Suppose $\st(x_4)$ separates $x_1$ and $x_2$. 
    
    Then $\{x_1,x_2,x_4\}$ is an FSIL.
\end{lem}

\begin{proof}
    First note that if $x_4$ is in a different connected component of $\Gamma$ to any of $x_1,x_2,x_3$,
    then it must be  in a different connected component from all three.
    If this is the case,
    then having $\st(x_4)$ separate $x_1$ and $x_2$ means that $x_1$ and $x_2$ must also be in different connected components.
    It then follows that $\{x_1,x_2,x_4\}$ is an FSIL.

    We now assume all vertices of the STIL are in the same connected component of $\Gamma$.
    It is clear that $(x_1,x_2,x_3\mids x_4)$ being a STIL implies that $(x_1,x_2\mids x_4)$ is a SIL.
    So we need to show that $(x_1,x_4\mids x_2)$ is a SIL
    (a symmetric argument will give the third SIL in the triple).
    We note that since $\st(x_4)$ separates $x_1$ and $x_2$, so will $\lk(x_4)$ and we  must have $\lk(x_1)\cap\lk(x_2)\subset\lk(x_4)$.

    Suppose that $(x_1,x_4\mids x_2)$ is not a SIL.
    Let $p=p_1,\ldots,p_k$ be a path of minimal length from $x_2$ to $\{x_1,x_4\}$ that avoids $\lk(x_1)\cap\lk(x_4)$.
    If $p_k=x_4$ then the path must intersect $\lk(x_1)\cap\lk(x_2)$, since the reverse of $p$ would be a path from $x_4$ to $x_2$ avoiding $\lk(x_1)\cap\lk(x_2)$.
    Now $\lk(x_1)\cap\lk(x_2)\subset\lk(x_4)$, and so we have $p_{k-1}\in\lk(x_1)\cap\lk(x_2)\cap\lk(x_4)$, contradicting our choice of path.
    If instead $p_k=x_1$, then the path intersects $\st(x_4)$.
    By the minimality of the path length, $p_{k-1}$ must be in $\lk(x_4)$.
    But then $p_{k-1} \in \lk(x_4)\cap\lk(x_1)$, a contradiction.
\end{proof}

\begin{lem}\label{lem:two separating gives SIL}
    Suppose we have that $\st(x_1)$ separates $x_2$ and $x_3$, and $\st(x_2)$ separates $x_1$ and $x_3$. Then $(x_1,x_2\mids x_3)$ is a SIL.

    In particular, if also $\st(x_3)$ separates $x_1$ and $x_2$, then $\{x_1,x_2,x_3\}$ is an FSIL.
\end{lem}

\begin{proof}
    Suppose $(x_1,x_2\mids x_3)$ is not a SIL.
    Let $p$ be a path $p_1,\ldots,p_k$ of minimal length from $x_3$ to $\{x_1,x_2\}$ that avoids $\lk(x_1)\cap\lk(x_2)$.
    Without loss of generality, we may assume $p_k=x_1$.
    Then the path $p$ passes through $\st(x_2)$.
    By minimality of the length of $p$, we must have $p_{k-1}\in \lk(x_2)$.
    But then $p_{k-1}\in\lk(x_1)\cap\lk(x_2)$ and we have a contradiction.
\end{proof}
We now use Lemmas~\ref{lem:PC with multiplier 4 gives FSIL} and
\ref{lem:two separating gives SIL} to eliminate cases that have
FSILs, reducing the number of cases necessary to prove the following
proposition.

\begin{prop}\label{prop:STIL no edge large}
    Suppose that $\Gamma$ contains a STIL $(x_1,x_2,x_3\mids x_4)$ in which no two vertices are connected by a single edge. 
    
    Then $\Out(W_\Gamma)$ is large.
\end{prop}

\begin{proof} By Lemmas~\ref{lem:PC with multiplier 4 gives FSIL} and
\ref{lem:two separating gives SIL}, up to relabelling $x_1,x_2,x_3$,
we may have the following automorphisms in a generating set of the
image of the factor map. For short-hand we denote $\chi^{x_i}_{x_j}$
by $\chi^i_j$.
\begin{enumerate}
    \item The STIL automorphisms: $\chi^{1}_{4}$, $\chi^{2}_{4}$, $\chi^{3}_{4}$. These are always present.
    \item $\chi^{1}_{2}$, if $\st(x_1)$ separates $x_2$ and $x_3$.
    \item $\chi^{2}_{3}$, if $\st(x_2)$ separates $x_1$ and $x_3$.
\end{enumerate}
Up to relabeling, we can either have only (1), only (1) and (2), or
all three.

Let $f\colon \outo{W_\Gamma} \to \outo{W_{\Gamma_4}}$ be the factor map for the STIL.
Following M\"{u}hlherr \cite{Muhlherr:presentation}, we can obtain a
presentation for $\Im(f)$ in each of the three cases listed above.
M\"{u}hlherr gives a presentation for the $\Aut^0(W_\Gamma)$, so when we pass to $\Out^0(W_\Gamma)$, we add the relators that the product of all partial conjugations with a given multiplier is trivial.
We thus have a presentation with generating set consisting of all partial conjugations $\chi^v_C$ where $C$ is a connected component of $\Gamma\setminus \st(v)$,
and set of relators as follows\footnote{M\"{u}hlherr lists four types of relators, two are not needed when restricting the generating set as we have done, and a third becomes trivial when working with outer automorphisms.}:
\begin{enumerate}
    \item[(a)] $[\chi^v_C , \chi^w_D] = 1$ whenever either $[v,w]=1$ or $(C\cup\{v\})\cap(D\cup\{w\})=\emptyset$;
    \item[(b)] if $C_1,\ldots,C_r$ is a complete set of all connected components of $\Gamma\setminus\st(v)$, then $\chi^v_{C_1}\cdots \chi^v_{C_r} = 1$;
    \item[(c)] $\left(\chi^v_C\right)^2=1$.
\end{enumerate}
Using \cite[Lemmas 5.2, 5.6 \& 5.9]{Muhlherr:presentation} the
subgroup $\Im(f)$ of $\Out^0(W_{\Gamma_4})$ will have presentations
with relators only of the three types listed above.

In the first case, with $\Im(f)$ generated only by automorphisms of
type (1), the image of $f$ has presentation:
$$
\langle \chi^{1}_{4}, \chi^{1}_{\{2,3\}},
\chi^{2}_{4},\chi^{2}_{\{1,3\}}, \chi^{3}_{4},\chi^{3}_{\{1,2\}}
\mid 
\chi^{1}_{4}\chi^{1}_{\{2,3\}},
\chi^{2}_{4}\chi^{2}_{\{1,3\}},
\chi^{3}_{4}\chi^{3}_{\{1,2\}}, 
\left(\chi^1_4\right)^2, \left(\chi^2_4\right)^2, \left(\chi^3_4\right)^2
\rangle.
$$
Using Tietze transformations, we can eliminate generators and see
that this subgroup is isomorphic to $\Z_2\ast\Z_2\ast\Z_2$, which is
virtually free. Thus, $\Out(W_\Gamma)$ surjects onto a virtually
free group and the result holds in this case.

In the second case, when $\Im(f)$ is generated by automorphisms of
types (1) and $(2)$, we have generators $\chi^1_2$ and $\chi^1_3$
instead of $\chi^1_{\{2,3\}}$. This yields the following
presentation:
\begin{align*}
\langle
\chi^{1}_{4}, \chi^1_2,\chi^1_3,
\chi^{2}_{4},\chi^{2}_{\{1,3\}},
\chi^{3}_{4},\chi^{3}_{\{1,2\}}
\mid
 \chi^1_2\chi^1_3\chi^{1}_{4},
&\,\chi^{2}_{4}\chi^{2}_{\{1,3\}},
\chi^{3}_{4}\chi^{3}_{\{1,2\}},
[\chi^1_2,\chi^3_4],\\
& 
[\chi^1_3,\chi^2_4],
\left(\chi^1_4\right)^2, \left(\chi^2_4\right)^2, \left(\chi^3_4\right)^2, 
 \left(\chi^1_2\right)^2
\rangle.
\end{align*}
Thus, $\Im(f)/\langle\langle\chi^1_2\rangle\rangle$ has the
following presentation:
\begin{align*}
\langle \chi^{1}_{4},\chi^1_3, \chi^{2}_{4},\chi^{2}_{\{1,3\}},
\chi^{3}_{4},\chi^{3}_{\{1,2\}} 
\mid
\chi^1_3\chi^{1}_{4},
\chi^{2}_{4}\chi^{2}_{\{1,3\}},
 \chi^{3}_{4}\chi^{3}_{\{1,2\}},
\left(\chi^1_4\right)^2, \left(\chi^2_4\right)^2, \left(\chi^3_4\right)^2, 
 [\chi^1_3,\chi^2_4]
\rangle.
\end{align*}
Eliminating redundant generators reduces this to:
$$
\langle
\chi^{1}_{4},
\chi^{2}_{4},
\chi^{3}_{4}
\mid
[\chi^1_4,\chi^2_4],
\left(\chi^1_4\right)^2, \left(\chi^2_4\right)^2, \left(\chi^3_4\right)^2
\rangle
$$
which is a presentation of the virtually free group $(\Z_2\times \Z_2)\ast\Z_2$.
Thus, $\outo{W_\Gamma}$ surjects onto a virtually
free group %(since $\Im(f)$ has a virtually free quotient) 
and so is large.

Finally, when $\Im(f)$ is generated by automorphisms of all three
types, we get:
\begin{align*}
\langle
\chi^{1}_{4}, \chi^1_2,\chi^1_3,
\chi^{2}_{4},\chi^{2}_{3},\chi^2_1,
\chi^{3}_{4},\chi^{3}_{\{1,2\}}
\mid
&\chi^1_2\chi^1_3\chi^{1}_{4},
\chi^{2}_{4}\chi^{2}_{3}\chi^2_1,
\chi^{3}_{4}\chi^{3}_{\{1,2\}},
[\chi^1_2,\chi^3_4],
 [\chi^1_3,\chi^2_4],
[\chi^2_3,\chi^1_4],\\
&
[\chi^2_1,\chi^3_4],
\left(\chi^1_4\right)^2, \left(\chi^2_4\right)^2, \left(\chi^3_4\right)^2,
\left(\chi^1_2\right)^2, \left(\chi^2_3\right)^2
\rangle.
\end{align*}
Killing $\chi^1_2$ and $\chi^2_1$ gives us a quotient with presentation
\begin{align*}
\langle
\chi^{1}_{4}, \chi^1_3,
\chi^{2}_{4},\chi^{2}_{3},
\chi^{3}_{4},\chi^{3}_{\{1,2\}}
\mid
\chi^1_3\chi^{1}_{4},
\chi^{2}_{4}\chi^{2}_{3},
&\chi^{3}_{4}\chi^{3}_{\{1,2\}},
[\chi^1_3,\chi^2_4],\\
&[\chi^2_3,\chi^1_4],
\left(\chi^1_4\right)^2, \left(\chi^2_4\right)^2, \left(\chi^3_4\right)^2
\rangle
\end{align*}
By Tietze moves, this is equivalent to
$$
\langle
\chi^{1}_{4},
\chi^{2}_{4},
\chi^{3}_{4}
\mid
[\chi^1_4,\chi^2_4],
\left(\chi^1_4\right)^2, \left(\chi^2_4\right)^2, \left(\chi^3_4\right)^2
\rangle
$$
which is also a presentation of $(\Z_2\times \Z_2)\ast\Z_2$. As
above, this implies that $\outo{W_\Gamma}$ is large, completing the
proof.\end{proof}

Propositions~\ref{prop:FSIL_large}, \ref{prop:STIL_edge_large}, and \ref{prop:STIL no edge large} combine to give Theorem~\ref{thm:STIL FSIL implies large}.
Theorem~\ref{thmspecial:main} follows immediately from Theorems~\ref{thm:virtually abelian}~and~\ref{thm:STIL FSIL implies large}.

\section{Graph products}\label{sec:graph products}

Recall that given a graph $\Gamma$ and a function $p$ which assigns to each vertex of $\Gamma$ a prime power, we can define the graph product $G(\Gamma,p)=G_\Gamma$ to be the group generated by the vertices of $\Gamma$, so that $v\in V(\Gamma)$ has order $p(v)$, and two generators commute if and only if they are adjacent in $\Gamma$. We say that a SIL $(x,y\mids z)$ is a \emph{non-Coxeter SIL} if $\max\{p(x), p(y)\}\ge 3$, or equivalently, the subgroup of $G_\Gamma$ generated by $x$ and $y$ is not virtually cyclic.

We now explain how to generalize the dichotomy from RACGs to these graph products, describing how to modify the proof to work in this more general setting.

\begin{proof}[Proof of Theorem~\ref{thmspecial:main graph prod}]
First suppose that $\Gamma$ contains no STIL, no FSIL, and no non-Coxeter SIL.
By Lemma~\ref{lem:partial conj not commute}, the absence of a non-Coxeter SIL implies that if $v\in \Gamma$ with $p(v)>2$, then $\chi^v_C$ is in the center of $\outo{G_\Gamma}$ for every component $C$ of $\Gamma\setminus\st(v)$.
Thus, Lemma~\ref{lem:gen set derived subgroup} still holds. That is, $\outo{G_\Gamma}'$ is generated by commutators of partial conjugations, since no partial conjugations $\chi^v_C$ with $p(v)>2$ contribute to the derived subgroup.
Indeed, it is generated by commutators of partial conjugations with $p(v)=2$. Since $\Gamma$ has no STIL or FSIL, Lemma~\ref{lem:derived subgroup abelian} still applies, telling us that $\outo{G_\Gamma}'$ is abelian. Thus, since $\outo{G_\Gamma}$ has finite abelianization, it is virtually abelian.

Now we move on to the cases where $\Gamma$ contains either a STIL, FSIL or non-Coxeter SIL.
For any full subgraph $\Gamma'$ of $\Gamma$, we consider the factor map, described in Section~\ref{sec:large case}:
\[f\colon \outo{G_\Gamma}\to\outo{G_{\Gamma'}}.\]
If $\Gamma$ contains an FSIL or a STIL with one edge, the arguments of Propositions~\ref{prop:FSIL_large}~and~\ref{prop:STIL_edge_large} still apply, since Theorem~\ref{thm:virtfree} holds for all free products of three finite groups.

If $\Gamma$ contains a non-Coxeter SIL $(x,y\mids z)$,
then take $\Gamma'$ to be the subgraph spanned by $\{x,y,z\}$.
By Theorem~\ref{thm:virtfree}, $\outo{G_{\Gamma'}}$ is virtually free and $\Im(f)$ contains a subgroup isomorphic to $\Z_{p(x)}\ast\Z_{p(y)}$.
Since either $p(x)>2$ or $p(y)>2$, $\Im(f)$ is virtually non-abelian free and so $\outo{G_\Gamma}$ is large.

The final case is when $\Gamma$ contains a STIL $(x_1, x_2, x_3\mids
y)$ with no edge connecting any $x_i$ to any $x_j$, for $1\leq i,j
\leq 3$. If $p(x_i)>2$ for any $1\le i\le 3$, then $\Gamma$ contains
a non-Coxeter SIL and $\outo{G_\Gamma}$ is large, by the above
argument. We thus suppose that $p(x_i)=2$ for $1\le i\le 3$. If
$\st(y)$ separates $x_i, x_j$, then Lemma~\ref{lem:PC with
multiplier 4 gives FSIL} implies that $\Gamma$ has an FSIL and again
$\outo{G_\Gamma}$ is large. Thus we may assume that the $\st(y)$
does not separate $x_i, x_j$ for $i,j\in\{1,2,3\}$. Consider the
factor map to $\outo{G_{\Gamma'}}$, where $\Gamma'$ is the subgraph
spanned by $\{x_1, x_2, x_3,y\}$. The image contains only partial
conjugations with multipliers $x_1, x_2, x_3$, and the action of these automorphisms is precisely the same as in the right-angled Coxeter case. Thus, this subgroup has
presentation exactly as in Proposition~\ref{prop:STIL no edge
large}\footnote{For an alternate argument, see \cite{Gilbert:Automorphisms}}, and so
$\outo{G_\Gamma}$ is large, completing the proof.

\end{proof}

\bibliographystyle{alpha}
\bibliography{bibliography_outRACG}

\end{document}